\documentclass[12pt]{amsart}
\usepackage{fullpage}

\usepackage{amsfonts,amssymb,amscd,amsmath,enumerate,verbatim,color,xcolor}
\usepackage[latin1]{inputenc}
\usepackage{amscd}
\usepackage{latexsym}
\usepackage{tikz}
\usepackage{graphicx} 

\usepackage{mathptmx}

\usepackage{hyperref}
\hypersetup{colorlinks,linkcolor=blue ,citecolor=blue, urlcolor=blue}
\def\NZQ{\mathbb}               

\def\ZZ{{\NZQ Z}}
\def\RR{{\NZQ R}}

\def\frk{\mathfrak}               

\def\Phi{{\frk N}}
\def\ab{{\mathbf a}}

\def\eb{{\mathbf e}}
\def\tb{{\mathbf t}}

\def\wb{{\mathbf w}}
\def\xb{{\mathbf x}}

\def\opn#1#2{\def#1{\operatorname{#2}}} 
\opn\gr{gr}

\newtheorem{Theorem}{Theorem}[section]
\newtheorem{Lemma}[Theorem]{Lemma}
\newtheorem{Corollary}[Theorem]{Corollary}
\newtheorem{Proposition}[Theorem]{Proposition}

\theoremstyle{definition}
\newtheorem{Remark}[Theorem]{Remark}

\newtheorem{Example}[Theorem]{Example}

\newtheorem{Conjecture}[Theorem]{Conjecture}

\let\epsilon\varepsilon
\let\phi=\varphi
\let\kappa=\varkappa
\opn\dis{dis}
\opn\height{height}
\opn\dist{dist}
\def\pnt{{\raise0.5mm\hbox{\large\bf.}}}

\opn\Lex{Lex}
\opn\conv{conv}

\title{Kempe equivalence and quadratic toric rings}
\author{Hidefumi Ohsugi and Akiyoshi Tsuchiya}
\address{Hidefumi Ohsugi,
	Department of Mathematical Sciences,
	School of Science,
	Kwansei Gakuin University,
	Sanda, Hyogo 669-1330, Japan} 
\email{ohsugi@kwansei.ac.jp}

\address{Akiyoshi Tsuchiya,
Department of Information Science,
Faculty of Science,
Toho University,
2-2-1 Miyama, Funabashi, Chiba 274-8510, Japan} 
\email{akiyoshi@is.sci.toho-u.ac.jp}

\dedicatory{Dedicated to the memory of J\"urgen Herzog}

\begin{document}

\begin{abstract}
Kempe equivalence is a classical and fundamental notion in graph coloring theory.
In the present paper we establish a connection between Kempe equivalence and quadratic stable set ring, which are toric rings associated to graphs.
In fact, we characterize when the stable set ring of a graph is quadratic by using Kempe equivalence.
As an application, we relate our theorem to the theory of perfectly contractile graphs, a hereditary subclass of perfect graphs introduced by Bertschi.
In particular, our characterization implies that the conjecture of Everett and Reed on perfectly contractile graphs entails the conjecture of the authors and Shibata on quadratic stable set rings.
Furthermore, we show that the stable set rings of several important subclasses of perfectly contractile graphs including weakly chordal graphs are quadratic.
Finally, we propose a new combinatorial conjecture characterizing perfectly contractile graphs purely in terms of Kempe equivalence on replication graphs.
\end{abstract}

\maketitle

\section{Introduction}
Let $G$ be a graph on the vertex set $[n]:=\{1,2,\ldots,n\}$ with the edge set $E(G)$.
A {\em $k$-coloring} $f$ of $G$ is a map $f: [n] \to [k]$ such that
$f(i) \neq f(j)$ for all $\{i,j\} \in E(G)$.
The smallest integer $\chi(G)$ for which $G$ admits a $\chi(G)$-coloring is called the \textit{chromatic number} of $G$.
Given a $k$-coloring $f$ of $G$, and integers $1 \le i < j \le k$,
let $H$ be a connected component of the subgraph of $G$ induced 
on the vertex set $f^{-1}(i) \cup f^{-1}(j)$.
Then we can obtain a new $k$-coloring $g$ of $G$ by setting
$$
g(x) =
\begin{cases}
    f(x) & x \notin H,\\
    i & x \in H \mbox{ and } f(x) =j,\\
    j & x \in H \mbox{ and } f(x) =i.\\
\end{cases}
$$
We say that $g$ is obtained from $f$ by a {\em Kempe switching}.
Two $k$-colorings $f$ and $g$ of $G$ are said to be {\em Kempe equivalent} if there exists
a sequence $f_0,f_1,\ldots,f_s$ of $k$-colorings of $G$ such that $f_0=f$, $f_s=g$,
and $f_i$ is obtained from $f_{i-1}$ by a Kempe switching.
For instance, any two $k$-colorings of a bipartite graph are Kempe equivalent.
Kempe switchings were introduced by Kempe in the false proof of the Four Color Theorem, but the idea turned out to be powerful in graph coloring theory. Since then, many researchers have studied Kempe switchings and Kempe equivalence.
See \cite{Moh} for an overview of Kempe equivalence.

The purpose of the present paper is to establish a connection between Kempe equivalence and the quadraticity of the \emph{stable set ring} $K[G]$ 
of $G$.
Given a subset $S \subset [n]$,
let $G[S]$ denote the induced subgraph of $G$ on the vertex set $S$.
A subset $S \subset [n]$ is called a {\em stable set} (or an {\em independent set}) of $G$
if $\{i,j\} \notin E(G)$ for all $i,j \in S$ with $i \neq j$.
Namely, a subset $S \subset [n]$ is stable if and only if $G[S]$ is an empty graph (i.e., a graph with no edges).
In particular, the empty set $\emptyset$ and any singleton $\{i\}$ with $i \in [n]$
are stable.
Denote $S(G)=\{S_1,\ldots,S_m\}$ the set of all stable sets of $G$.
Given a subset $S \subset [n]$, we associate the $(0,1)$-vector $\rho(S)=\sum_{j \in S} \eb_j$. Here $\eb_j$ is the $j$th unit coordinate vector in $\RR^n$. For example, $\rho(\emptyset)=(0,\ldots,0) \in \RR^n$.
Let $K[\tb,s]:=K[t_1,\ldots,t_n,s]$ be the polynomial ring in $n+1$ variables  over a field $K$.
Given a nonnegative integer vector 
$\ab=(a_1,\ldots,a_n) \in  \ZZ_{\geq 0}^n$, we write 
$\tb^{\ab}:=t_1^{a_1} t_2^{a_2}\cdots t_n^{a_n} \in K[\tb,s]$.
The \textit{stable set ring} of $G$ is 
\[
K[G]:=K[\tb^{\rho(S_1)} s,\ldots, \tb^{\rho(S_m)} s] \subset K[\tb,s]. 
\]
We regard $K[G]$ as a homogeneous algebra by setting each $\deg (\tb^{\rho(S_i)} s)=1$.
Note that $K[G]$ is a toric ring.
Let $K[\xb]:=K[x_1,\ldots,x_m]$ denote the polynomial ring in $m$ variables over  $K$ with each $\deg(x_i)=1$.
The \textit{stable set ideal}
$I_G$ of $G$ is the kernel of the surjective homomorphism 
$\pi:K[\xb] \to K[G]$ defined by $\pi(x_i)=\tb^{\rho(S_i)}s$ for $1 \leq i \leq m$.
Note that $I_G$ is a toric ideal, and generated by homogeneous binomials.
The toric ring $K[G]$ is called \textit{quadratic}
if $I_G$ is generated by quadratic binomials.
We say that ``$I_G$ is generated by quadratic binomials'' even if $I_G =\{0\}$ (or equivalently, $G$ is complete).
It is easy to see that a homogeneous binomial
$x_{i_1} \cdots x_{i_r} - x_{j_1} \cdots x_{j_r} \in K[\xb]$ belongs to $I_G$
if and only if 
$\bigcup_{\ell=1}^r S_{i_\ell}=\bigcup_{\ell=1}^r S_{j_\ell}$ 
as multisets.
See, e.g., \cite{HHO} for details on toric rings and toric ideals.
The following is a connection between Kempe equivalence and the quadraticity of the stable set ring $K[G]$.

\begin{Theorem}
\label{quadgene}
    Let $G$ be a graph on the vertex set $[n]$.
    Then $K[G]$ is quadratic if and only if, for every replication graph $H$ of any induced subgraph of $G$ and every $k \ge \chi(H)$,
    all $k$-colorings of $H$ are Kempe equivalent (see Section~\ref{sect:kempe} for the definition of replication graphs).
\end{Theorem}

Kempe equivalence has also played a key role in the study of \emph{perfectly contractile graphs}, a hereditary subclass of perfect graphs introduced by Bertschi~\cite{Bert}.
In particular, Bertschi~\cite{Bert} proved that if a graph $G$ is perfectly contractile, 
then for every $k \geq \chi(G)$, all $k$-colorings of $G$ are Kempe equivalent. 
By combining this observation with Theorem \ref{quadgene}, we can derive applications to a conjecture for quadratic stable set rings.

A graph $G$ is called \textit{perfect} if every induced subgraph $H$ of $G$ satisfies $\chi(H)=\omega(H)$, where $\omega(H)$ is the maximum cardinality of a clique in $H$. Perfect graphs were introduced by Berge \cite{Berge}.
A \textit{hole} is an induced cycle of length $\geq 5$ and an \textit{antihole} is the complement of a hole.
A hole is called \textit{odd},
if its length is odd.
An \textit{odd antihole} is the complement of 
an odd hole.
The celebrated Strong Perfect Graph Theorem due to 
Chudnovsky, Robertson, Seymour, and Thomas~\cite{CRST} asserts that 
a graph is perfect if and only if it contains no odd holes and no odd antiholes.

Bertschi~\cite{Bert} further defined a hereditary subclass of perfect graphs.
An \textit{even pair} in a graph $G$ is a pair of non-adjacent vertices of $G$ such that the length of all induced paths between them is even.
Contracting a pair of vertices $\{x,y\}$ in a graph $G$ means removing $x$ and $y$ and adding a new vertex $z$ with edges to every neighborhood of $x$ or $y$.
A graph $G$ is called \textit{even-contractile} if there exists a sequence $G_0,\ldots,G_k$ of graphs satisfying the following:
\begin{enumerate}[(i)]
    \item $G=G_0$;
    \item each $G_i$ is obtained from $G_{i-1}$ by contracting an even pair of $G_{i-1}$;
    \item $G_k$ is a complete graph.
\end{enumerate}
A graph $G$ is called \textit{perfectly contractile} if every induced subgraph of $G$ is even-contractile.
Every perfectly contractile graph is perfect, and, in particular,
Meyniel graphs, perfectly orderable graphs, and clique-separable graphs 
are perfectly contractile~\cite{Bert}.
Moreover, every weakly chordal graph is perfectly contractile~\cite{HHM}.
Everett and Reed 
conjectured a characterization of perfectly contractible graphs
via a list of forbidden induced subgraphs.
An \textit{odd} (resp.~\textit{even}) \textit{prism} is a graph consisting of two disjoint triangles with three disjoint induced paths of odd (resp.~even) length between them.

\begin{Conjecture}[\cite{ER}]
\label{conj:first}
A graph is perfectly contractile if and only if it contains no odd holes, no antiholes and no odd prisms as induced subgraphs.
\end{Conjecture}

In \cite{LMR}, it was shown that
any perfectly contractile graph
contains no odd holes, no antiholes and no odd prisms as induced subgraphs.
Moreover, Conjecture~\ref{conj:first} holds 
for planar graphs \cite{LMR}, 
dart-free graphs \cite{dartfree}, and even prism-free graphs \cite{prismfree}.
However, Conjecture~\ref{conj:first} is still open.

The authors and Shibata~\cite{OST} proposed an algebraic analogue of Conjecture~\ref{conj:first}
in terms of the stable set ring.

\begin{Conjecture}[{\cite[Conjecture 0.2]{OST}}]
\label{conj:second} Let $G$ be a perfect graph. 
Then the stable set ring $K[G]$ is quadratic
if and only if $G$ contains no odd holes, no antiholes and no odd prisms as induced subgraphs.

\end{Conjecture}

In~\cite{OST}, it was shown that a graph
$G$ contains no odd holes, no antiholes and no odd prisms as induced subgraphs if $K[G]$ is quadratic.
Moreover, it was shown that $K[G]$ is quadratic when $G$ is a Meyniel graph, 
a perfectly orderable graph, or a clique-separable graph.

As an application of Theorem~\ref{quadgene}, 
we show that Conjecture~\ref{conj:first} implies Conjecture~\ref{conj:second}.

\begin{Theorem}
\label{thm:conj}
If Conjecture~\ref{conj:first} is true, so is Conjecture~\ref{conj:second}.
\end{Theorem}

Furthermore, Theorem~\ref{quadgene} yields the quadraticity of $K[G]$
for several important subclasses of perfectly contractile graphs.

\begin{Theorem}
\label{thm:app}
For any graph $G$ belonging to one of the following subclasses of 
perfectly contractile graphs,
$K[G]$ is quadratic.
\begin{itemize}
    \item[(a)] 
dart-free graphs with no odd holes, no antiholes, and no odd prisms{\rm ;}
    \item[(b)] 
even prism-free graphs with no odd holes, no antiholes, and no odd prisms{\rm ;}
    \item[(c)]
weakly chordal graphs{\rm ;}
    \item[(d)]
Meyniel graphs \cite{OST}{\rm ;}
    \item[(e)]
    perfectly orderable graphs \cite{OST}.
\end{itemize}
\end{Theorem}
Finally, combining Conjecture \ref{conj:second} with Theorem \ref{quadgene}, we propose a new purely combinatorial conjecture characterizing perfectly contractile graphs.

\begin{Conjecture}
\label{conj:comb}
Let $G$ be a perfect graph. Then $G$ is perfectly contractile if and only if for every replication graph $H$ of an arbitrary induced subgraph of $G$ and every $k \geq \chi(H)$, all $k$-colorings of $H$ are Kempe equivalent.
\end{Conjecture}

The structure of the paper is as follows:
In Section~\ref{sect:kempe}, we explore the relationship between Kempe equivalence and stable set rings, 
in particular providing a system of generators of the defining ideal of $K[G]$ in terms of colorings 
(Theorem~\ref{binomial_color}) and proving Theorem~\ref{quadgene}.
Section~\ref{sect:app} discusses applications of Theorem~\ref{quadgene} 
and contains the proofs of Theorems~\ref{thm:conj} and~\ref{thm:app}.


\section{Kempe equivalence and quadratic stable set rings}
\label{sect:kempe}
In this section, we characterize when the stable set ring $K[G]$ of a (not necessarily perfect) graph is quadratic.
In particular, we prove Theorem~\ref{quadgene}.
First, we discuss a relationship between monomials of degree $k$ in $K[\xb]$ and $k$-colorings.

Given a graph $G$ on the vertex set $[n]$, and 
$\ab = (a_1,\ldots,a_n) \in \ZZ_{\ge 0}^n$,
let $G_\ab$ be the graph obtained from $G$ by replacing each vertex $i \in [n]$ 
with a complete graph $G^{(i)}$ of $a_i$ vertices (if $a_i =0$, then just delete the vertex $i$),
and joining all vertices $x \in G^{(i)}$ and $y \in G^{(j)}$ such that $\{i,j\}$ is an edge of $G$.
In particular, if $\ab =(1,\ldots,1)$, then $G_\ab = G$.
If $\ab = {\bf 0}$, then $G_\ab$ is the null graph (a graph without vertices). 
In addition, if $\ab$ is a $(0,1)$-vector, then $G_\ab$ is an induced subgraph of $G$.
If $\ab$ is a positive vector, then $G_\ab$ is called a 
\textit{replication graph} of $G$.
In general, $G_\ab$ is a replication graph of an induced subgraph of $G$.

\begin{Lemma}
\label{lem:color-stable}
Let $G$ be a graph on the vertex set $[n]$ and 
let $f$ be a $k$-coloring of $G_{\ab}$ where 
$\ab \in \ZZ^n_{\geq 0}$.
Then for each $\ell=1,2,\ldots,k$,
 the set $
S=\{j \in [n] : G^{(j)} \cap f^{-1}(\ell) \ne \emptyset\}
$
is a stable set of $G$.
\end{Lemma}

\begin{proof}
    Suppose that $\{j_1,j_2\}$ is an edge of $G$ for some $j_1 , j_2 \in S$.
    Then $G^{(j_1)} \cap f^{-1}(\ell) \ne \emptyset$ and $G^{(j_2)} \cap f^{-1}(\ell) \ne \emptyset$.
Since $G^{(j_1)} \cap f^{-1}(\ell) \ne \emptyset$ and $G^{(j_2)} \cap f^{-1}(\ell) \ne \emptyset$ hold,
there exist $v_1 \in G^{(j_1)} \cap f^{-1}(\ell)$ and $v_2 \in G^{(j_2)} \cap f^{-1}(\ell)$.
Since $\{j_1,j_2\}$ is an edge of $G$, $\{v_1,v_2\}$ is an edge of $G_\ab$.
However, since $v_1,v_2 \in f^{-1}(\ell)$, this contradicts that $f$ is a $k$-coloring of $G_\ab$.
Thus $S$ is a stable set of $G$.
\end{proof}

Let $G$ be a graph on the vertex set $[n]$ with stable sets $S(G)=\{S_1,\ldots,S_m\}$
and let $f$ be a $k$-coloring of $G_\ab$ where 
$\ab \in \ZZ_{\ge 0}^n$.
From Lemma~\ref{lem:color-stable} we can associate $f$ with a monomial 
\begin{eqnarray}
\varphi_k (f) :=x_{i_1} \cdots x_{i_k} \in K[\xb], \label{ctom}
\end{eqnarray}
where
$S_{i_\ell} =  \{j \in [n] : G^{(j)} \cap f^{-1}(\ell) \ne \emptyset\}$
for $\ell=1,2,\ldots, k$.

\begin{Lemma}
\label{monomial_color}
Let $G$ be a graph on the vertex set $[n]$.
Then, for any positive integer $k$, 
$\varphi_k$ given in {\rm (\ref{ctom})} is a
surjective map from 
$\{ f : f \mbox{ is a } k\mbox{-coloring of } G_\ab \mbox{ for some } \ab \in \ZZ_{\ge 0}^n  \}$
to the set of all monomials of degree $k$ in $K[\xb]$.
Moreover, $\varphi_k (f) = \varphi_k (g)$
if and only if $f$ is obtained from $g$ by either exchanging colors or
exchanging the coloring of vertices in each clique $G^{(j)}$ in $G_\ab$.
\end{Lemma}

\begin{proof}
Let $M=x_{i_1} \cdots x_{i_k}$ be a monomial in $K[\xb]$.
Let $\ab=(a_1,\ldots,a_n)$ be the integer vector defined by
$a_p = |\{ \ell : p \in S_{i_\ell}\}|$.
Then $S_{i_1},\ldots, S_{i_k}$ defines
a $k$-coloring $f$ of $G_\ab$ as follows:
if $S_{i_\ell}$ contains $j$, then let $f(x) = \ell$ for one of $x \in G^{(j)}$
which is not colored yet.
Then $\varphi_k(f) = M$.
Thus $\varphi_k$ is surjective.
This coloring is unique up to the choices of a vertex from $G^{(j)}$.    
\end{proof}

Since a homogeneous binomial $x_{i_1} \cdots x_{i_k} - x_{j_1} \cdots x_{j_k} \in K[\xb]$ belongs to $I_G$ if and only if 
$\bigcup_{\ell=1}^k S_{i_\ell}=\bigcup_{\ell=1}^k S_{j_\ell}$ 
as multisets (i.e., $|\{ \ell : p \in S_{i_\ell}\}|=|\{ \ell : p \in S_{j_\ell}\}|$ for any $1 \le p \le n$), we can give a system of generators of $I_G$ as follows.

\begin{Theorem}
\label{binomial_color}
Let $F =x_{i_1} \cdots x_{i_k} - x_{j_1} \cdots x_{j_k} \in K[\xb]$ be a homogeneous binomial.
Then $F \in I_G$ if and only if 
there exists $\ab \in \ZZ_{\ge 0}^n$ such that
$\varphi_k(f) = x_{i_1} \cdots x_{i_k}$
and 
$\varphi_k(g) =x_{j_1} \cdots x_{j_k}$ 
for some 
$k$-colorings $f$ and $g$ of $G_\ab$.
In particular, one has 
\[
I_G= \langle \phi_k(f)- \phi_k(g) : \mbox{$f$ and $g$ are k-colorings of $G_\ab$ with $\ab \in \ZZ^d_{\geq 0}$ and $k \geq \chi(G_{\ab})$}  \rangle.
\]
\end{Theorem}



Now, we can prove Theorem~\ref{quadgene}.

\begin{proof}[Proof of Theorem~\ref{quadgene}]
(only if) 
The proof is by contradiction.
Suppose that $K[G]$ is quadratic.
Let $f$ and $g$ be two $k$-colorings of $G_\ab$ with $\ab  \in \ZZ_{\ge 0}^n$
which are not Kempe equivalent.
From Theorem~\ref{binomial_color}, the binomial
$$F := 
\varphi_k (f) - \varphi_k(g) =x_{i_1} \cdots x_{i_k} - x_{j_1} \cdots x_{j_k}
$$
belongs to $I_G$.
By the hypothesis, $F$ is generated by quadratic binomials in $I_G$.
Then the theory of binomial ideals \cite[Lemma 3.8]{HHO} guarantees that
there exists an expression
\begin{equation}
F = \sum_{r=1}^s {\bf x}^{\wb_r} (x_{p_r} x_{q_r} - x_{p_r'} x_{q_r'}),
\label{tenkai}
\end{equation}
where each $x_{p_r} x_{q_r} - x_{p_r'} x_{q_r'}$ belongs to $I_G$.
Suppose that $s$ is minimal among all the sums and all pairs of $k$-colorings of $G_\ab$
which are not Kempe equivalent.
Since $x_{i_1} \cdots x_{i_k}$ must appear in the right-hand side of (\ref{tenkai}),
we may assume that $x_{i_1} \cdots x_{i_k} = {\bf x}^{\wb_1} x_{p_1} x_{q_1}$.
Then there exist $\ell$ and $\ell'$ such that $x_{i_\ell} x_{i_{\ell'}}=x_{p_1} x_{q_1}$, say, $ x_{i_1} x_{i_2} = x_{p_1} x_{q_1}$.
Note that $\varphi_2(f_1) = x_{i_1} x_{i_2}$,
where $f_1$ is the restriction of $f$
to $G' = G_\ab[f^{-1} (1) \cup f^{-1} (2)]$.
Moreover since $x_{p_1} x_{q_1} - x_{p_1'} x_{q_1'}
=
x_{i_1} x_{i_2} - x_{p_1'} x_{q_1'}
$ belongs to $I_G$,
$\varphi_2(f_2) = x_{p_1'} x_{q_1'}$ for some 2-coloring $f_2$ of the graph $G'$.
%
Since $G'$ has a 2-coloring, it is a bipartite graph.
Then $f_1$ and $f_2$ are Kempe equivalent
(See, e.g., \cite[Proposition~2.1 (a)]{Moh}).
Let $f'$ be a coloring of $G_\ab$ defined by
$$
f'(x) =
\begin{cases}
    f_2(x) &  f(x) \in \{1,2\},\\
         f(x) &\mbox{otherwise.}
\end{cases}
$$
It then follows that $f$ and $f'$ are Kempe equivalent.
Moreover we have
$$
x_{i_1} \cdots x_{i_k} - 
{\bf x}^{\wb_1} (x_{p_1} x_{q_1} - x_{p_1'} x_{q_1'})
= x_{p_1'}  x_{p_2'} x_{i_3} \cdots x_{i_k} =\varphi_k (f').
$$
Since
$$
\varphi_k (f') - \varphi_k (g)
=
x_{p_1'}  x_{p_2'} x_{i_3} \cdots x_{i_k}
-x_{j_1} \cdots x_{j_k}
= \sum_{r=2}^s {\bf x}^{\wb_r} (x_{p_r} x_{q_r} - x_{p_r'} x_{q_r'}),
$$
$f'$ and $g$ are Kempe equivalent by the hypothesis on $s$.
Since $f$ and $f'$ are Kempe equivalent,
this contradicts that $f$ and $g$ are not Kempe equivalent.

(if) 
 Suppose that for every $\ab \in \ZZ_{\ge 0}^n$,
    all $k$-colorings of $G_\ab$ are Kempe equivalent.
Let $F= x_{i_1} \cdots x_{i_k} - x_{j_1} \cdots x_{j_k} \in I_G$ be a homogeneous binomial of degree $k \ge 3$.
Suppose that $F$ is not generated by quadratic binomials in $I_G$.

From Theorem~\ref{binomial_color},
there exist $k$-colorings $f$ and $g$ of $G_\ab$
such that $\varphi_k (f) = x_{i_1} \cdots x_{i_k}$
and $\varphi_k (g) = x_{j_1} \cdots x_{j_k}$,
where $\ab=(a_1,\ldots,a_n) \in \ZZ_{\ge 0}^n$ is given by
$a_p = |\{ \ell : p \in S_{i_\ell}\}|\  (= |\{ \ell : p \in S_{j_\ell}\}|)$.
By the hypothesis, $f$ and $g$ are Kempe equivalent.
Let $f_0,f_1,\ldots,f_s$ be a sequence of $k$-colorings of $G_\ab$ such that $f_0=f$, $f_s=g$,
and $f_i$ is obtained from $f_{i-1}$ by a Kempe switching.
Suppose that $s$ is minimal among all binomials which are not generated by quadratic binomials in $I_G$.

Suppose that the Kempe switching from $f$ to $f_1$ is obtained by
a connected component $H$ of the induced subgraph 
$G_\ab [f^{-1}(\mu) \cup f^{-1}(\eta)]$ by setting
$$
f_1(x) =
\begin{cases}
    f(x) & x \notin H,\\
    \mu & x \in H \mbox{ and } f(x) =\eta,\\
    \eta & x \in H \mbox{ and } f(x) =\mu.\\
\end{cases}
$$
Let $f|_{G'}$ and $f_1|_{G'}$ be the restrictions of  $f$ and $f_1$ to $G':= G_\ab [f^{-1}(\mu) \cup f^{-1}(\eta)]$, respectively.
Since $f|_{G'}$ and $f_1|_{G'}$ are 2-colorings of ${G'}$, 
there exists a quadratic binomial $x_{i_\mu} x_{i_\eta} - x_{\sigma} x_{\tau} \in I_G$ where
$S_\sigma = \{j \in [n] : G^{(j)} \cap f_1^{-1}(\mu) \ne \emptyset\}$ and $S_\tau = \{j \in [n] : G^{(j)} \cap f_1^{-1}(\eta) \ne \emptyset\}$.
Then 
$$
F = \frac{x_{i_1} \cdots x_{i_k}}{x_{i_\mu} x_{i_\eta}} (x_{i_\mu} x_{i_\eta} - x_{\sigma} x_{\tau} )
+ F',
$$ 
where 
$$F'=  \frac{x_{i_1} \cdots x_{i_k}}{x_{i_\mu} x_{i_\eta}}
x_{\sigma} x_{\tau} - x_{j_1} \cdots x_{j_k}
\in I_G.$$
Moreover $F' = \varphi_k (f_1) - \varphi_k(g)$ for
the $k$-colorings $f_1$ and $g$ of $G_\ab$.
By the hypothesis on $s$, $F'$ is generated by quadratic binomials $F_1, \ldots, F_t $ in $I_G$.
Hence $F$ is generated by quadratic binomials $x_{i_\mu} x_{i_\eta} - x_{\sigma} x_{\tau}, F_1, \ldots, F_t $ in $I_G$, a contradiction.
\end{proof}

By using Theorem~\ref{quadgene}, we give an example of even-contractile graph $G$ such that $K[G]$ is not quadratic.
\begin{Example}
Let $G$ be the following graph:
	
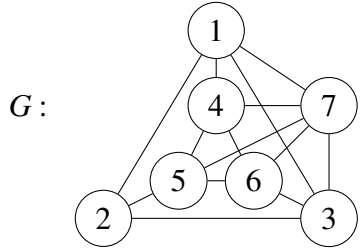
\begin{center}
		\begin{tikzpicture}
\node (w) at (-2.5,0.5) {$G:$}; 
\node[draw, shape=circle] (a) at (0,1.5) {$1$}; 
\node[draw, shape=circle] (b) at (-1.5,-1) {$2$}; 
\node[draw, shape=circle] (c) at (1.5,-1) {$3$}; 
\node[draw, shape=circle] (d) at (0,0.5) {$4$}; 
\node[draw, shape=circle] (e) at (-0.5,-0.5) {$5$}; 
\node[draw, shape=circle] (f) at (0.5,-0.5) {$6$}; 
\node[draw, shape=circle] (g) at (1.5,0.5) {$7$};

\draw (a)--(b);
\draw (b)--(c);
\draw (c)--(a);
\draw (d)--(e);
\draw (e)--(f);
\draw (f)--(d);
\draw (a)--(d);
\draw (b)--(e);
\draw (c)--(f);

\draw (g)--(a);
\draw (g)--(c);
\draw (g)--(d);
\draw (g)--(e);
\draw (g)--(f);
\end{tikzpicture}
	\end{center}
 Then $G$ is even-contractile. Indeed, in the following graphs, each pair of the black vertices is an even pair and each $G_i$ is obtained from $G_{i-1}$ by contracting the even pair of $G_{i-1}$. 
\begin{center}
		\begin{tikzpicture}
  \node (w) at (-2.5,0.5) {$G_0:$}; 
\node[draw, shape=circle] (a) at (0,1.5) {};
\node[draw, shape=circle,fill] (b) at (-1.5,-1){};
\node[draw, shape=circle] (c) at (1.5,-1){};
\node[draw, shape=circle] (d) at (0,0.5){};
\node[draw, shape=circle] (e) at (-0.5,-0.5){};
\node[draw, shape=circle] (f) at (0.5,-0.5){};
\node[draw, shape=circle,fill] (g) at (1.5,0.5){};

\draw (a)--(b);
\draw (b)--(c);
\draw (c)--(a);
\draw (d)--(e);
\draw (e)--(f);
\draw (f)--(d);
\draw (a)--(d);
\draw (b)--(e);
\draw (c)--(f);

\draw (g)--(a);
\draw (g)--(c);
\draw (g)--(d);
\draw (g)--(e);
\draw (g)--(f);
\node (w1) at (3.5,0.5) {$G_1:$}; 
\node[draw, shape=circle] (a1) at (6.5,1.5){};
\node[draw, shape=circle] (b1) at (5,-1){};
\node[draw, shape=circle,fill] (c1) at (8,-1){};
\node[draw, shape=circle,fill] (d1) at (6.5,0.5){};
\node[draw, shape=circle] (e1) at (6,-0.5){};
\node[draw, shape=circle] (f1) at (7,-0.5){};

\draw (a1)--(b1);
\draw (b1)--(c1);
\draw (c1)--(a1);
\draw (d1)--(e1);
\draw (e1)--(f1);
\draw (f1)--(d1);
\draw (a1)--(d1);
\draw (b1)--(e1);
\draw (c1)--(f1);

\draw (b1)--(d1);
\draw (b1)--(f1);
\end{tikzpicture}
\end{center}
\begin{center}
		\begin{tikzpicture}
   \node (w) at (-2.5,0.5) {$G_2:$}; 
\node[draw, shape=circle,fill] (a) at (0,1.5) {};
\node[draw, shape=circle] (b) at (-1.5,-1){};
\node[draw, shape=circle] (d) at (0,0.5){};
\node[draw, shape=circle,fill] (e) at (-0.5,-0.5){};
\node[draw, shape=circle] (f) at (0.5,-0.5){};

\draw (a)--(b);
\draw (d)--(e);
\draw (e)--(f);
\draw (f)--(d);
\draw (a)--(d);
\draw (b)--(e);
\draw (b)--(d);
\draw (b)--(f);

 \node (w) at (3.5,0.5) {$G_3:$}; 
\node[draw, shape=circle] (b1) at (5,-1){};
\node[draw, shape=circle] (d1) at (6.5,0.5){};
\node[draw, shape=circle] (e1) at (6,-0.5){};
\node[draw, shape=circle] (f1) at (7,-0.5){};

\draw (d1)--(e1);
\draw (e1)--(f1);
\draw (f1)--(d1);
\draw (b1)--(e1);

\draw (b1)--(d1);
\draw (b1)--(f1);
\end{tikzpicture}
\end{center}

\noindent
 It then follows that for every $k \geq \chi(G) =4$, all $k$-colorings of $G$ are Kempe equivalent (see Proposition~\ref{ecKEMPE}).
 On the other hand, for $\ab=(1,1,1,1,1,1,0)$, $G_\ab$ is an odd prism which has 3-colorings $f$ and $g$ defined by
 \[
f(i)=\begin{cases}
     1 & \mbox{if $i \in \{1,5\}$} \\
       2 & \mbox{if $i \in \{2,6\}$} \\
         3 & \mbox{if $i \in \{3,4\}$} 
\end{cases}
\mbox{ and }
g(i)=\begin{cases}
     1 & \mbox{if $i \in \{1,6\}$} \\
       2 & \mbox{if $i \in \{2,4\}$} \\
         3 & \mbox{if $i \in \{3,5\}$}. 
\end{cases}
 \]
Clearly, $f$ and $g$ are not Kempe equivalent (as observed in \cite{vdH}). 
Hence from Theorem \ref{quadgene}, $K[G]$ is not quadratic.
In fact, setting 
\[S_1=\{1,5\},\ S_2=\{2,6\}, \ S_3=\{3,4\},\ S_4=\{1,6\}, \ S_5=\{2,4\},\ S_6=\{3,5\}, \]
each $S_i$ is a stable set of $G$ and $F=x_1x_2x_3-x_4x_5x_6$ is a binomial in $I_G$ which is not generated by quadratic binomials in $I_G$.
\end{Example}

\begin{Remark}
It is known that $G$ is perfect if and only if 
the initial ideal of $I_G$ is generated by squarefree monomials
with respect to any reverse lexicographic order \cite{OHcompressed}.
In such a case, $I_G$ is generated by the binomials of degree $\le n+1$
from \cite[Proposition~13.15 (ii)]{Stu}.
Hence, if $G$ is perfect, then we can replace ``$k \ge \chi(H)$"
by ``$\chi(H) \le k \le n+1$" in Theorem~\ref{quadgene}.
\end{Remark}

\section{Applications to perfectly contractile graphs}
\label{sect:app}

In this section, 
we consider applications of Theorem~\ref{quadgene} and prove Theorems~\ref{thm:conj} and \ref{thm:app}.
The following result is a motivation for Bertschi to define perfectly contractile graphs.

\begin{Proposition}[\cite{Bert}]
\label{ecKEMPE}
Let $G$ be an even-contractile graph.
Then $\chi(G)=\omega(G)$ and for every $k \ge \chi(G)$,
all $k$-colorings are Kempe equivalent.
\end{Proposition}

By using this proposition and Theorem~\ref{quadgene}, we have the following.
\begin{Corollary}
\label{Ga contractile}
Let $\mathcal{C}$ be a class of some perfectly contractile graphs.
Suppose that, for any $G \in \mathcal{C}$ and $\ab \in \ZZ_{\ge 0}^n$,
the graph
$G_\ab$ belongs to $\mathcal{C}$.
Then $K[G]$ is quadratic for every $G \in \mathcal{C}$.
\end{Corollary}

\begin{proof}
Let $G \in \mathcal{C}$ and $\ab \in \ZZ_{\ge 0}^n$.
Since $G_\ab \in \mathcal{C}$ is perfectly contractile, 
all $k$-colorings of $G_\ab$ are Kempe equivalent by
Proposition~\ref{ecKEMPE}.
From Theorem~\ref{quadgene}, $K[G]$ is quadratic.   
\end{proof}

A pair of adjacent vertices $\{x,y \}$ of $G$ is called an {\em adjacent twin} in $G$
if $N(x) \setminus \{y\} = N(y) \setminus \{x\}$,
where $N(x)$ is the set of all neighbours of $x$ in $G$.
By using Corollary ~\ref{Ga contractile}, we can prove Theorem~\ref{thm:conj}.

\begin{proof}[Proof of Theorem~\ref{thm:conj}]
Suppose that Conjecture~\ref{conj:first} is true, that is,
the class $\mathcal{A}$ of graphs that contain no odd holes, no antiholes, and no odd prisms 
coincides with the set of all perfectly contractile graphs.
It is enough to show 
that the stable set ring $K[G]$ is quadratic
for every $G \in \mathcal{A}$.

Let $G \in \mathcal{A}$.
From Corollary~\ref{Ga contractile},
it is enough to show that, for any $\ab \in \ZZ_{\ge 0}^n$,
we have $G_\ab \in \mathcal{A}$.
Since the class $\mathcal{A}$ is hereditary,
we may assume that $\ab$ is a positive vector.
Note that, if $x,y \in G^{(i)}$ with $x\ne y$, then $\{x,y\}$ is an adjacent twin.
Since odd holes, antiholes, and odd prisms have no adjacent twins,
$G_\ab$ has one of such graphs as an induced subgraph if and only if so does $G$.
Hence we have $G_\ab \in \mathcal{A}$.
Thus 
the stable set ring $K[G]$ is quadratic
for every $G \in \mathcal{A}$.
\end{proof}

We apply Corollary \ref{Ga contractile} to several important subclasses of perfectly contractile graphs to prove Theorem~\ref{thm:app}.
The \textit{dart} is the graph on the vertex set $\{1,2,3,4,5\}$
and the edge set 
$$\{ \{1,2\},\{2,3\},\{1,5\},\{2,5\},\{3,5\},\{4,5\} \}.$$
A graph is called \textit{weakly chordal} if it has no holes and no antiholes.
A graph is called \textit{Mayniel} if any odd cycle of length $\ge 5$
has at least two chords.
Let $G$ be a graph on the vertex set $[n]$.
An ordering $(i_1,\ldots,i_n)$ of $[n]$ is called {\em perfect}
if, for any induced (ordered) subgraph $H$ of $G$,
the number of colors used 
by the greedy coloring algorithm on $H$ is 
$\chi(H)$.
It is known \cite[Theorem~1]{perfectlyorderable} that
a vertex ordering $<$ of a graph $G$ is perfect if and only if $G$ contains no induced $P_4$
$abcd$ such that $a<b$ and $d<c$.
A graph is called {\em perfectly orderable} if there exists a perfect ordering $(v_1,\dots,v_n)$ of 
the vertex set of $G$.

\begin{proof}[Proof of Theorem~\ref{thm:app}]
From Corollary~\ref{Ga contractile},
it is enough to show that, for any $\ab \in \ZZ_{\ge 0}^n$,
if $G$ satisfies one of (a) -- (e), then so does $G_\ab$.
Since all of the classes are hereditary,
we may assume that $\ab$ is positive.

The dart, holes, antiholes, odd/even prisms, and odd cycles of length $\ge 5$ having exactly one chord
are forbidden graphs for the classes (a), (b), (c) and (d),
and have no adjacent twins.
Thus $G_\ab$ has one of such graphs as an induced subgraph if and only if so does $G$.   

Let $G$ be a perfectly orderable graph with a perfect ordering $(i_1,\ldots,i_n)$.
Note that an induced $P_4$ has no adjacent twins.
Hence any vertex ordering of $G_\ab$ such that $x < y$ whenever
$x \in G^{(i)}$, $y \in G^{(j)}$ with $i < j$ is a perfect ordering of $G_\ab$.
Thus $G_\ab$ is perfectly orderable.
\end{proof}

\subsection*{Acknowledgment}
The first author was partially supported by JSPS KAKENHI 24K00534 and the second author was partially supported by JSPS KAKENHI 22K13890 and 26K00618.

\end{document}